\documentclass[11pt,reqno,bookmarksnumbered,raiselinks,breaklinks,colorlinks,plainpages]{amsart}

\usepackage{amsmath, amsthm, amscd, amsfonts, amssymb, graphicx, color}

\DeclareMathOperator*{\argmin}{arg\,min}

\renewcommand{\epsilon}{\varepsilon}

\newcommand\ubd{\overline{\mbox{\rm dim}}_{\rm B}\,} 
\newcommand\lbd{\underline{\mbox{\rm dim}}_{\rm B}\,} 
\newcommand\da{\underline{\mbox{\rm dim}}_{\,\theta}} 

\newcommand\uda{\overline{\mbox{\rm dim}}_{\,\theta}} 

\newtheorem{theo}{Theorem}[section]

\newtheorem{lem}[theo]{Lemma}
\newtheorem{prop}[theo]{Proposition}

\newtheorem{defi}[theo]{Definition}

\newtheorem{rmk}[theo]{Remark}
\newtheorem{question}{Question}

\title[Intermediate Dimensions of Complementary Sets]{Intermediate Dimensions of Complementary Sets}

\author{Nicolás Angelini}
\address{Departamento de Matemática, Facultad de Ciencias Físico Matemáticas y Naturales, Universidad Nacional de San Luis,\\
and Instituto de Matemática Aplicada San Luis (IMASL, CONICET), San Luis, Argentina.}
\email{nicolas.angelini.2015@gmail.com}

\author{Úrsula Molter}
\address{Departamento de Matemática, Facultad de Ciencias Exactas y Naturales, Universidad de Buenos Aires,\\
and Instituto de Investigaciones Matemáticas Luis A. Santaló (IMAS, UBA-CONICET), Buenos Aires, Argentina.}
\email{umolter@gmail.com}

\subjclass[2020]{28A80, 28A75}
\keywords{Fractals, Intermediate dimensions, Complementary sets}

\begin{document}

\begin{abstract}
  Given a positive, non-increasing sequence $a$ with finite sum equal to $1$, we consider the family of all closed subsets of $[0,1]$ whose complementary open intervals have lengths given by a rearrangement of the sequence $a$. We study the full range of possible $\theta$-intermediate dimensions of these sets and, under suitable assumptions on the sequence, we show that this range forms a closed interval, whose endpoints we compute explicitly. This paper fills a gap in the literature concerning the dimensional properties of complementary sets.
\end{abstract}
\maketitle

\section{Introduction}

The intermediate dimensions $\underline{\dim}_\theta$ and $\overline{\dim}_\theta$ form a continuum of dimensions that were introduced in \cite{FFK} as a way to interpolate between the Hausdorff and box-counting dimensions; see Definition \ref{adef} for the precise definition. These dimensions have been extensively studied in various contexts, such as in \cite{Brownian1} for stochastic processes, \cite{BFF} and \cite{angelini} for results related to projections of sets, \cite{BFF2} for elliptic polynomial spirals, and more recently in \cite{angelini2}, where these dimensions were extended to measures. These dimensions are continuous (except perhaps at $\theta=0$) and increase as a function of $\theta$, and the following relations hold:
\[\dim_H E \leq \da E \leq \uda E \leq \ubd \quad \text{and} \quad \da E \leq \lbd E,\]
where $E \subset \mathbb{R}^d$ is compact, and $\dim_H$ and $\dim_B$ represent the Hausdorff and box dimensions, respectively. Moreover, \cite{Banaji-R22} provides a characterization of a function $h(\theta)$ as the dimension function associated with a set $E$. These dimensions were further generalized in \cite{Banaji} to more general settings.

Given $a=\{a_n\}$, a positive non-increasing sequence with a finite sum equal to $1$, the class $\mathcal{C}_a$ is defined as the family of all closed subsets of $[0,1]$ whose complement in $[0,1]$ is a union of disjoint open intervals of lengths given by the entries of $a$. The sets in $\mathcal{C}_a$ are called \textit{complementary sets} of $a$. They can be countable, for example, a sequence of points, and uncountable, for example, Cantor sets. So it is natural to ask about the possible dimensions of the sets belonging to $\mathcal{C}_a$.

In the case of Hausdorff dimension, the problem was first studied in 1954 by Besicovitch and Taylor in \cite{BeyTa} and was later studied more deeply in \cite{cabrelli2010classifying} and \cite{garcia2007dimension}. In fact, it was proved that the set of possible values for the Hausdorff dimension of the sets in $\mathcal{C}_a$ is the closed interval $[0,\dim_H C_{a}]$, where $C_{a}$ denotes the Cantor set associated with $a$, see section \ref{Comp sets} for the precise definition. A similar result for the Packing measure was obtained in \cite{hare2013sizes} by Hare, Mendivil, and Zuberman.

  For the case of the box dimensions, in \cite[Chapter 3.2]{falconer1997techniques} it is shown that there exist numbers $s \leq t$, depending only on the sequence $a$, such that every complementary set $E \in \mathcal{C}_a$ has lower and upper box dimensions $
\lbd E = s $ and  $\ubd E = t$.
In particular, these values depend only on $a$ and not on the specific rearrangement of its gaps.

 More recently, Garcia, Hare and Mendivil studied the problem for Assouad and lower Assouad dimensions. They computed the maximum and minimum possible values for the Assouad dimension of complementary sets and, under certain assumptions about the decay of the sequence $a$, proved that the set of attainable values for the Assouad dimension of sets in $\mathcal{C}_a$ is, in fact, also a closed interval.

In this paper, we study the problem from the perspective of $\theta$-intermediate dimensions. Specifically, we compute the $\theta$-intermediate dimension of the countable set $D_a$ (see Section \ref{Comp sets} for the precise definition) and, under natural conditions on the decay of the sequence $a$, we are able to compute the $\theta$-intermediate dimension of the Cantor set $C_a$. 

Following these computations, we establish that the minimum and maximum possible values for the $\theta$-intermediate dimension of any complementary set $E \in \mathcal{C}_a$ are given by the $\theta$-intermediate dimensions of $D_a$ and $C_a$, respectively. More formally, for all $\theta \in [0, 1]$, we have the following inequalities:
$$\uda D_a \leq \uda E \leq \uda C_a.$$ Similarly, the analogous result holds for the lower $\theta$-intermediate dimension.

Finally, we address the natural question of whether all possible values of $\theta$-intermediate dimensions for sets in $\mathcal{C}_a$ are attained within the intervals $[\underline{\dim}_\theta D_a, \underline{\dim}_\theta C_a]$ and $[\overline{\dim}_\theta D_a, \overline{\dim}_\theta C_a]$. We prove that this is indeed the case under certain natural geometric conditions on the sequence $a$.

Let $s_n$ denote the average length of the closed intervals at the $n$-th step in the construction of the Cantor set $C_a$. Our hypotheses are motivated by known results for Assouad and lower Assouad dimensions: the Assouad dimensions of complementary sets form a closed interval if the sequence $a$ is doubling, while the lower Assouad dimensions form a closed interval if $2 < \inf_n s_n/s_{n+1}$ (see \cite{garcia2018assouad}). These two conditions—which together require that the ratios $s_n/s_{n+1}$ be bounded away from both 2 and infinity—ensure that the intermediate dimensions of complementary sets also form a closed interval. Specifically, fixed $\theta$, we show that under these conditions,
\[
\{ \underline{\dim}_\theta E : E \in \mathcal{C}_a \} = [\underline{\dim}_\theta D_a, \underline{\dim}_\theta C_a]
\]
and
\[
\{ \overline{\dim}_\theta E : E \in \mathcal{C}_a \} = [\overline{\dim}_\theta D_a, \overline{\dim}_\theta C_a].
\]
That is, fixed $\theta\in[0,1]$, for every $t\in[\da D_a,\da C_a]$ (respectively for the upper intermediate dimension), there exists a set $E \in \mathcal{C}_a$, that depends on $\theta$ and $t$, whose lower (respectively upper) $\theta$-intermediate dimension equals exactly $t$.

As previously discussed, complementary sets have already been studied extensively with respect to Hausdorff, packing, box-counting, Assouad, and lower Assouad dimensions. 

This paper complements those results by providing a detailed analysis of the intermediate dimensions of sets in 
$\mathcal{C}_a$  for an arbitrary sequence 
$a$, characterizing the full range of possible values.

\section{Preliminaries}

\subsection{Definitions}

Given a non-empty set $A \subset \mathbb{R}$, we write $|A|$ for the diameter of the set.\\ We write \(a_n \asymp b_n\) to mean that there exist constants \(c_1, c_2 > 0\), independent of \(n\), such that 
\(c_1 b_n \leq a_n \leq c_2 b_n\) for all \(n\).

The Hausdorff dimension of $A$ will be denoted by $\dim_H A$, and $\lbd A$ and $\ubd A$ will denote the lower and upper box-counting dimensions of $A$, respectively. The reader can refer to \cite{FBook} for more information on these dimensions.

This document concerns the $\theta$-intermediate dimensions, which were introduced in \cite{FFK} and are defined as follows.

\begin{defi}\label{adef}
Let $F \subseteq \mathbb{R}$ be bounded. For $0 \leq \theta \leq 1$, we define the \emph{lower $\theta$-intermediate dimensions} of $F$ by
\begin{align*}
 \da F =  \inf \big\{& s \geq 0 : \ \forall\ \epsilon >0, \ \text{and all }\ \delta_0>0, \text{there}\ \exists\ 0<\delta\leq \delta_0, \ \text{and } \\
 & \{U_i\}_{i \in I}\  : F \subseteq \cup_{i \in I} U_i \ : \delta^{1/\theta} \leq  |U_i| \leq \delta \ \text{and }\ \sum_{i \in I} |U_i|^s \leq \epsilon  \big\}.
\end{align*}

Similarly, we define the \emph{upper $\theta$-intermediate dimensions} of $F$ by
\begin{align*}
\uda F =  \inf \big\{& s \geq 0 : \ \forall\ \epsilon >0, \ \text{there}\ \exists\ \delta_0>0, \ : \ \forall\ 0<\delta\leq \delta_0, \ \text{there}\ \exists \\
 & \{U_i\}_{i \in I}\  : F \subseteq \cup_{i \in I} U_i \ : \delta^{1/\theta} \leq  |U_i| \leq \delta \ \text{and }\ \sum_{i \in I} |U_i|^s \leq \epsilon  \big\}.
\end{align*}
\end{defi}

For all $\theta \in (0,1]$ and a compact set $A \subset \mathbb{R}$, we have $\dim_H A \leq \uda A \leq \ubd A$ and similarly with the lower $\theta$ dimension, replacing $\ubd$ with $\lbd$. This spectrum of dimensions has the property of being continuous for $\theta \in (0,1]$, and sometimes it is also continuous at $\theta=0$. The reader can find more information in \cite{FFK}.

\begin{rmk}
    It is straightforward to note that the value of $\da F$ and $\uda F$ remains unchanged if the covering condition $\delta^{1/\theta} \le |U_i| \le \delta$ is replaced by $c_1\delta^{1/\theta} \le |U_i| \le c_2\delta$ for any fixed constants $c_1, c_2 > 0$.
\end{rmk}

We also need to recall the definitions of the Assouad dimension and the lower Assouad dimension. Given a non-empty set $E \subset \mathbb{R}$, $N_r(E)$ will denote the minimum number of closed balls of radius $r$ needed to cover $E$.

The \emph{Assouad dimension} of $E$ is defined as 
\begin{align*}
    \dim_A E= \inf\{ \alpha : & \text{ there exist constants } c,\rho>0 \text{ such that }\\
    & \sup_{x \in E} N_r (B(x,R) \cap E) \leq c\left(\frac{R}{r}\right)^\alpha \text{ for all } 0<r<R<\rho \}.
\end{align*}
In the same way, the \emph{lower Assouad dimension} of $E$ is defined as
\begin{align*}
    \dim_L E= \sup\{ \alpha : & \text{ there exist constants } c,\rho>0 \text{ such that }\\
    & \inf_{x \in E} N_r (B(x,R) \cap E) \geq c\left(\frac{R}{r}\right)^\alpha \text{ for all } 0<r<R<\rho \}.
\end{align*}

The following inequalities hold for compact sets:
$$\dim_L E \leq \dim_H E \leq \da E\leq \uda E\leq \ubd E\leq \dim_A E.$$

A key inequality involving these dimensions, which we will use later, is given in the following result. 
\begin{prop} [Corollary 2.8 \cite{Banaji-R22}] \label{banajis bound}
  
If $\dim_L E<\dim_A E$, then for all $\theta \in (0,1]$,
\[
     \uda E \geq \frac{\theta\dim_A E (\ubd E - \dim_L E) + \dim_L E (\dim_A E - \ubd E)}{\theta(\ubd E - \dim_L E) + (\dim_A E-\ubd E)}.
\]
The same holds replacing $\uda$ and $\ubd$ with $\da$ and $\lbd$ everywhere.
\end{prop}

\section{Complementary Sets}\label{Comp sets}
Let $a=\{a_n\}$ be a positive, non-increasing, summable sequence with sum $1$. Let $I_a=[0,1]$. $\mathcal{C}_a$ will denote the family of all closed sets $E$ contained in $I_a$ that are of the form $E=I_a \setminus \displaystyle\cup_{j \geq 1} U_j$, where $\{U_j\}$ is a disjoint family of open intervals contained in $I_a$ such that $|U_j|=a_j$ for all $j$. The sets in $\mathcal{C}_a$ are called the complementary sets of $a$. Note that if $\mathbf{a}=\{\mathbf{a}_n\}$ is any rearrangement of the sequence $\{a_n\}$, then $\mathcal{C}_a=\mathcal{C}_{\mathbf{a}}$.

An important set in $\mathcal{C}_a$ is the \emph{countable} set, defined as follows:
$$D_a=\left\{\sum_{i=k}^{\infty} a_i\right\}_{k\in\mathbb{N}}.$$ 
As the reader notes, $D_a$ is a countable, decreasing set of points with decreasing gaps given by the elements of the sequence $a$.

Another set of great importance in this family is the non-increasing Cantor set associated with the sequence $a$, which is constructed inductively in the following way. In the first step, remove from $I_a$ an open interval of length $a_1$, leaving two closed intervals $I_{1,1}$ and $I_{1,2}$. Having performed the $k$-th step, we have $2^k$ closed intervals $I_{k,1},...,I_{k,2^{k}}$ contained in $I_a$. Then, from each interval $I_{k,j}$, remove an open interval of length $a_{2^k +j-1}$ to obtain two closed intervals $I_{k+1,2j -1}$ and $I_{k+1,2j}$. Finally, define the Cantor set as $$C_a := \displaystyle\bigcap_{k=1}^{\infty}\bigcup_{j=1}^{2^{k}} I_{k,j}.$$

Given a non-increasing, summable sequence $a=\{a_n\}$, define
\[
\displaystyle x_n(a)=\sum_{i=n}^{\infty}a_i
\quad \text{ and }\quad s_n(a)=2^{-n} x_{2^n}(a).\]

When the sequence $a$ is fixed and clear from context, we will simply write $s_n$ instead of $s_n (a)$.

Note that since $a$ is non increasing we have the estimate  
\[s_{n+1} < |I_{n,j}| < s_{n-1}
\quad \text{for all $1 \leq j \leq 2^n$,}\]
where $I_{n,j}$ are the closed intervals of step $n$ in the construction of the Cantor set $C_{a}$.

The Hausdorff dimension of sets in $\mathcal{C}_a$ was studied in \cite{BeyTa}, \cite{cabrelli2010classifying}, and \cite{garcia2007dimension}. The following proposition holds:

\begin{prop}\label{hausdorff cantor}
Let $a = \{a_n\}$ be a positive summable sequence, and let $C_a$ be the decreasing Cantor set associated with $a$. Then
\[
\dim_H C_{a} = \liminf_{n \to \infty} \frac{\log n}{-\log (x_n/n)}.
\]
Furthermore, any set $E \in \mathcal{C}_{a}$ satisfies $\dim_H E \leq \dim_H C_{a}$. Conversely, given any $s \in [0, \dim_H C_{a}]$, there exists a set $E \in \mathcal{C}_a$ such that $\dim_H E = s$.
\end{prop}

Moreover, it was shown in \cite[Proposition 3.3]{garcia2007dimension} that $\dim_H C_a = \lbd C_a$.

The box dimensions of sets in $\mathcal{C}_a$ were carefully studied in \cite[Chapter 3.2]{falconer1997techniques}. There, Falconer shows that the upper and lower box dimensions of sets in $\mathcal{C}_a$ depend only on the sequence $a$, and not on the rearrangement. That is, given any complementary set, its lower box dimension equals that of every other complementary set, and the same holds for the upper box dimension. Furthermore, for any $E \in \mathcal{C}_a$,
\[
\liminf_{n \to \infty} \frac{\log n}{-\log a_n} \leq \lbd E \leq \ubd E \leq \limsup_{n \to \infty} \frac{\log n}{-\log a_n},
\]
and the box dimension exists if and only if these limits coincide.

The upper box dimension of $C_a$ (and therefore of any other complementary set in $\mathcal{C}_a$) is given by 
\[
\limsup_{n \to \infty} \frac{\log n}{-\log (x_n/n)},
\]
which, as shown in \cite{tricot1994curves}, is equal to 
\[
\limsup_{n \to \infty} \frac{\log n}{-\log a_n}.
\]
This equivalence does not hold for the lower box dimension. In fact, the inequality
\[
\liminf_{n \to \infty} \frac{\log n}{-\log a_n} \leq \liminf_{n \to \infty} \frac{\log n}{-\log (x_n/n)}
\]
can be strict; see, for instance, \cite[Proposition 5]{cabrelli2004hausdorff}.
The Assouad dimension and the lower Assouad dimension of sets in \( \mathcal{C}_a \) were studied in \cite{garcia2018assouad}. There, the authors prove that, given a non-increasing summable sequence \( a = \{a_n\} \),
\[
\dim_A C_a = \limsup_{n \to \infty} \left( \sup_k \frac{n \log 2}{\log(s_k / s_{k+n})} \right),
\]
and
\[
\dim_L C_a = \liminf_{n \to \infty} \left( \inf_k \frac{n \log 2}{\log(s_k / s_{k+n})} \right).
\]
Moreover, it was shown that \( \dim_A D_a = 0 \) or \( 1 \), and if \( E \) is any complementary set in \( \mathcal{C}_a \), then
\[
\dim_A C_a \leq \dim_A E \leq \dim_A D_a,
\]
and
\[
0 = \dim_L D_a \leq \dim_L E \leq \dim_L C_a.
\]

Furthermore, if the sequence \( a \) is doubling, i.e., there exists a constant \( c > 0 \) such that \( a_n \leq c \, a_{2n} \) for all \( n \), then
\[
\{\dim_A E: E \in \mathcal{C}_a\} = [\dim_A C_a, 1],
\]
and if instead of doubling, we assume that \( \inf \frac{s_n}{s_{n+1}} > 2 \), then
\[
\{\dim_L E: E \in \mathcal{C}_a\} = [0, \dim_L C_a].
\]
The present paper fills a gap in the theory of complementary sets by proving equivalent results through the 
$\theta-$intermediate dimension.
\section{$\theta-$intermediate dimension of $C_a$ and $D_a$}
In this section, we compute the lower and upper $\theta$-intermediate dimensions of the set $D_a$.

Under natural assumptions on the decay of the ratios $s_n$, we also determine the lower and upper intermediate dimensions of the Cantor set $C_a$. To this end, for a fixed scale $s_n$, we consider coverings by sets whose diameters lie within the admissible range determined by $\theta$.

Specifically, we aim to minimize the quantity  
\(
\frac{-m \log 2}{\log s_m},
\) 
where the index $m$ ranges over all the values satisfying 
\(
s_n^{1/\theta} \leq s_m \leq s_n.
\)  
In other words, the goal is to minimize this expression over all admissible diameters $s_m$ that may be used to cover the set at scale $s_n$, in accordance with the $\theta$-intermediate dimension constraints.

This leads naturally to the introduction of two parameters, which will aid in analyzing the minimal value of the above expression within the admissible range.
\begin{defi}\label{rho}
Let $a = \{a_n\}_{n \in \mathbb{N}}$ be a non-increasing summable sequence, let $\theta \in (0,1]$ and $r \in \mathbb{N}$. We define $\gamma_{\theta,a}(r)$ and $\rho_{\theta,r}(r)$ as follows:
\[
\gamma_{\theta,a}(r) := \max \left\{ m \in \mathbb{N} : s_m \geq s_r^{1/\theta} \right\},
\]
and
\[
\rho_{\theta,a}(r) := \argmin_{r \leq m \leq \gamma(r)+1} \left\{ s_m^{1/m} \right\},
\]
that is, $\rho_{\theta,a}(r)$ is the integer between $r$ and $\gamma_{\theta,a}(r)+1$ for which the value $s_m^{1/m}$ is minimal.
\end{defi}

A consequence of the following result is that the best possible covering strategy 
for $C_a$ uses sets of some fixed radius within the allowed range $[\delta^{1/\theta}, \delta]$.
\begin{theo}\label{dimension Cantor}
Let $a = \{a_n\}_{n \in \mathbb{N}}$ be a non increasing summable sequence. Suppose that \(\lim_{n\to\infty} (s_n/s_{n+1})^{1/n}=1\). Then, for all $\theta \in (0,1]$, we have
\begin{equation}\label{Th inciso 1} \overline{\dim}_\theta C_a=
    \limsup_{n \to \infty} \frac{-\rho_{\theta,a}(n) \log 2}{ \log s_{\rho_{\theta,a}(n)}}.
\end{equation}

\end{theo}

\begin{proof}
  Our first goal is to show that the limsup on the right-hand side of \eqref{Th inciso 1} is less than or equal to the $\theta$-intermediate dimension of $C_a$.

To do this, let $\theta\in(0,1]$ and note that if $\uda C_a = 0$, then Proposition \ref{banajis bound} implies that $\ubd C_a = 0$, and therefore
\[
\limsup_{n \to \infty} \frac{-n \log 2}{\log(s_n)} = 0,
\]
which in turn implies that the right-hand side of \eqref{Th inciso 1} is also equal to $0$.
Now, suppose that $\uda C_a > 0$.

For simplicity of notation, we will write \( \rho(n) \) and \( \gamma(n) \) instead of \( \rho_{\theta,a}(n) \) and \( \gamma_{\theta,a}(n) \), respectively.

Let $\delta_0 > 0$ and let $K$ be such that $s_K < \delta_0$, and fix $s' < \displaystyle\limsup_{n \rightarrow \infty} \frac{-\rho(n)\log(2)}{\log(s_{\rho(n)})}$.

\begin{equation}\label{s'}\text{Then there exists $k' \geq K$ such that}\quad   
s' < \frac{ -\rho(k')\log(2)}{\log(s_{\rho(k')})} .
\end{equation}

Note that by the definition of $\rho$ in (\ref{rho}), we have
\[
\frac{-\rho (k') \log (2)}{\log ( s_{\rho (k')} )} \leq \frac{-m \log (2)}{\log ( s_{m} )}\quad \text{for all $k' \leq m \leq \gamma(k') + 1$}.
\]

Let $\delta = s_{k'} < \delta_0$, and let $\{U_i\}$ be any finite cover of $C_a$ such that $\delta^{1/\theta} \leq |U_i| \leq \delta$ for all $i$. Then we have $(s_{k'})^{1/\theta} \leq |U_i| \leq s_{k'}$ for all $i$.

For each $i$, let $t(i) \in \mathbb{N}$ be such that $s_{t(i)+1} < |U_i| \leq s_{t(i)}$. Note that $k' \leq t(i) \leq \gamma(k')$.

Then $U_i$ can intersect at most 4 intervals from step $t(i)-1$ in the construction of $C_a$. Indeed, if it were to intersect 5 intervals from step $t(i)-1$, it would necessarily contain at least one interval from step $t(i)-2$, say $I_{t(i)-2,j}$ for some $j$, and hence:
\[
s_{t(i)} \geq |U_i| \geq |I_{t(i)-2,j}| \geq s_{t(i)-1} > s_{t(i)},
\]
since $s_n$ is strictly decreasing.

As a consequence, $U_i$ can intersect at most 8 intervals from step $t(i)$. Therefore, for $j \geq t(i)$, the set $U_i$ can intersect at most
\[
8 \, 2^{j-t(i)} = 16 \, 2^j \, 2^{-(t(i)+1)} \leq 16 \, 2^j \, (s_{t(i)+1})^{s'} \leq 16 \, 2^j \, |U_i|^{s'}
\]
intervals from step $j$.

Finally, let $j$ be sufficiently large so that $s_j < |U_i|$ for all $i$. Since $\{U_i\}_i$ is a cover of $C_a$, it must intersect all the $2^j$ intervals from step $j$, and hence:
\[
2^j \leq \sum_i 16 \, 2^j \, |U_i|^{s'},
\]
which implies that any such cover satisfies:
\[
\sum_i |U_i|^{s'} \geq \frac{1}{16}.
\]

Therefore, $s' \leq \uda C_a$. Letting 
\[
s' \to \limsup_{n \to \infty} \frac{-\rho(n)\log(2)}{\log(s_{\rho(n)})}
\quad \text{we have} \quad \limsup_{n \to \infty} \frac{-\rho(n)\log(2)}{\log(s_{\rho(n)})} \leq \uda C_a.
\]

For the opposite inequality, let  
\[
t > d := \limsup_{n \to \infty} \frac{-\rho(n)\log(2)}{\log(s_{\rho(n)})},
\]
and let \( \delta > 0 \) be arbitrarily small. Let \( k \) be such that \( s_{k+1} \leq \delta \leq s_k \), and suppose that \( \delta^{1/\theta} \leq s_{\rho(k)} \leq \delta \). Since the lengths of the closed intervals at step \( \rho(k)+1 \) in the construction of \( C_a \) are less than \( s_{\rho(k)} \), we can cover \( C_a \) with \( 2^{\rho(k)+1} \) intervals of length \( s_{\rho(k)} \).  

Then  
\[
\inf_{\{U_i\} : \delta^{1/\theta} \leq |U_i| \leq \delta} \sum_i |U_i|^t \leq 2^{\rho(k)+1} s_{\rho(k)}^t < C' < \infty,
\]
for some constant \( C' > 0 \), by our choice of \( t \). The above inequality implies that \( \uda C_a \leq t \) and the result follows by letting \( t \to d \).

Now, if \( s_{\rho(k)} < \delta^{1/\theta} \) or \( \delta < s_{\rho(k)} \), then either \( s_{\rho(k)-1} \) or \( s_{\rho(k)+1} \) belongs to \( [\delta^{1/\theta}, \delta] \). Proceeding as before, we obtain  
\[
\uda C_a \leq \limsup_{n \to \infty} \frac{-\rho(n)\log(2)}{\log(s_{\rho(n)-1})}
\quad \text{or} \quad
\uda C_a \leq \limsup_{n \to \infty} \frac{-\rho(n)\log(2)}{\log(s_{\rho(n)+1})}.
\]  

But, since we are assuming that \( (s_n / s_{n-1})^{1/n} \) tends to \( 1 \) as \( n \to \infty \), we have  
\begin{align*}
\limsup_{n \to \infty} \frac{-\rho(n)\log(2)}{\log(s_{\rho(n)+1})}  
&= \limsup_{n \to \infty} \frac{-\rho(n)\log(2)}{\log(s_{\rho(n)+1} / s_{\rho(n)}) + \log(s_{\rho(n)})}  
\\ &= \limsup_{n \to \infty} \frac{-\rho(n)\log(2)}{\log(s_{\rho(n)})},
\end{align*}
and the same holds for \( s_{\rho(n)-1} \), which completes the proof.
\end{proof}
Note that it is not necessary to compute the lower \( \theta \)-intermediate dimension of \( C_a \), since the equality \( \dim_H C_a = \underline{\dim}_B C_a \) always holds.

In the next theorem, we compute the intermediate dimensions of the countable set $D_a$. 
We note that these dimensions are as small as possible, attaining the minimal value 
permitted by Proposition \ref{banajis bound}.

\begin{theo}\label{dimension de suceciones} 
Let \(a=\{a_n\}_{n\in\mathbb{N}}\) be a non increasing summable sequence. Then for all \(\theta\in [0,1]\), we have
\begin{equation}\label{Theo inciso 11}
\uda D_{\textbf{a}} = \frac{\theta \, \ubd D_{\textbf{a}} }{1 - (1-\theta)\, \ubd D_{\textbf{a}}}
\end{equation}
and  
\[
\da D_{\textbf{a}} = \frac{\theta \, \lbd D_{\textbf{a}} }{1 - (1-\theta)\, \lbd D_{\textbf{a}} }.
\]
\end{theo}

\begin{proof}\ We will prove the result for the upper intermediate dimensions. The analogous statement for the lower intermediate dimensions holds with suitable adjustments, replacing the use of the $\limsup$ with that of the $\liminf$.

    First we will prove that the right hand side of equality \eqref{Theo inciso 11} is an upper bound to $\uda D_a$.
    We have \begin{equation}\label{dimb sucecion} B:=\ubd D_{\textit{\textbf{a}}}=\displaystyle\limsup_{n\to\infty}\frac{\log (n)}{-\log(x_n/n)}, \end{equation}

with $x_n=\displaystyle\sum_{i=n}^{\infty}a_i$.    
If $B=0$ there is nothing to prove and if $B=1$ then by Proposition \ref{banajis bound}, $\uda D_{\textit{\textbf{a}}}=1$, so we will suppose that $0<B<1$. 

Let $0<\epsilon<1-B$.  
From \eqref{dimb sucecion} we have that there exists $N$, such that $x_n\leq (1/n)^{\frac{1-(B+\epsilon)}{B+\epsilon}}$.
Let $\delta>0$ be sufficiently small such that $N\leq \lceil \delta^{-(B+\epsilon)(s+\theta(1-s))}\rceil:=K.$ 
To cover $\{x_n\}$ we will cover $\{x_i : 1\leq i \leq K\}$ with $K$ intervals of length $\delta$ and the remaining points $\{x_i : K<i\}\subset [0, x_K]$ with 
$$\left\lceil \frac{x_K}{\delta^{\theta}}\right\rceil\leq \frac{(1/K)^{\frac{(1-(B+\epsilon))}{B+\epsilon}}}{\delta^{\theta} }+1$$ intervals of length $\delta^{\theta}$. 

Hence we have a cover $\mathcal{U}$ of $D_a$ such that $\delta\leq |U|\leq \delta^{\theta}$ for all $U\in \mathcal{U}$ and 
\begin{align*}
    \sum_{U\in\mathcal{U}} |U|^{s}&\leq K\delta^{s}+\left(\frac{(1/K)^{\frac{(1-(B+\epsilon))}{B+\epsilon}}}{\delta^{\theta}} +1\right)\delta^{\theta s}\\
    &\leq (\delta^{-(B+\epsilon)(s+\theta(1-s))}+1)\delta^{s} + \delta^{\theta(s-1)}\delta^{(s+\theta(1-s))(1-(B+\epsilon)))} + \delta^{\theta s}\\
    &=I+I\!I+I\!I\!I
\end{align*}
respectively.

The term $I\!I\!I$ will tend to $0$ as $\delta$ tends to $0$  if and only if $\theta s >0$.

The terms $I$ and $I\!I$ will tend to $0$ as $\delta$ tends to $0$ if and only if $s>\frac{\theta (B+\epsilon)}{1-(1-\theta)(B+\epsilon)}.$

This implies that 
$$\uda D_a\leq \frac{\theta (B+\epsilon)}{1-(1-\theta)(B+\epsilon)},$$
and the result follows letting $\epsilon\to 0.$ 

For the reverse inequality, note that since \( \dim_L D_a = 0 \) and \( \dim_A D_a = 1 \), the result follows directly from Proposition \ref{banajis bound}.

\end{proof}

\section{Maximum and minimum dimensions}

In this section, we will show that \( D_{a} \) and \( C_{a} \) attain the minimum and maximum among all possible values for intermediate dimensions. We first make the following straightforward remark.

\begin{rmk}\label{lemma cubrimiento}
Let $c > 0$. If $\{U_i\}_{i=1}^{m}$ is a collection of $m$ closed intervals such that \(
\sum_{i=1}^{m} |U_i| \leq c \, m,\)
then the union $\bigcup_{i=1}^{m} U_i$ can be covered by at most $2m$ closed intervals of diameter $c$.
\end{rmk}

\begin{lem}
    Let $a=\{a_n\}$ be a summable non-increasing sequence and let $E\in \mathcal{C}_a$ be any complementary set. Then 
    $$\uda D_{a} \leq \uda E \quad\text{ and }\quad \da D_{a} \leq \da E.$$ 
\end{lem}
\begin{proof}\
     Using that all complementary sets in $\mathcal{C}_a$ share the same lower and upper box dimensions, that is, there exist numbers $s \leq t$ depending only on $a$ such that $\lbd E = s$ and $\ubd E = t$ for every $E \in \mathcal{C}_a$, together with Proposition \ref{banajis bound}
 we have
 \begin{align*}
     \uda D_a &=\frac{\theta\, \ubd D_a}{1-(1-\theta)\,\ubd D_a}
     = \frac{\theta \, \ubd E}{1 -(1-\theta)\,\ubd E}\\
     &\leq \frac{\theta\dim_A E \,\ubd E  }{\dim_A E-(1-\theta)\,\ubd E}\\
     &\leq \uda E.
 \end{align*}

 The proof for the lower $\theta$-intermediate dimensions is similar.
 \end{proof}

\begin{lem}
Let \( a = \{a_n\} \) be a summable sequence such that  
\(
\displaystyle\lim_{n \to \infty} \left( \frac{s_n}{s_{n+1}} \right)^{1/n} = 1.
\)  
Then any set \( E \in \mathcal{C}_a \) satisfies  
\[
\uda E \leq \uda C_a \quad \text{and} \quad \da E \leq \da C_a.
\]
\end{lem}
\begin{proof}\
The lower $\theta-$intermediate dimension case follows directly from the fact that $\da C_a=\lbd C_a=\lbd E$ for all $\theta.$
 Now, we will prove the upper intermediate dimension part. By Lemma~\ref{dimension Cantor}, we have
\[
\overline{\dim}_\theta C_a =
\limsup_{n \to \infty} \frac{-\rho(n) \log 2}{\log s_{\rho(n)}}.
\]
Let \( t > \uda C_a \), so that there exists \( K \in \mathbb{N} \) such that \( 2^{\rho(k)} s_{\rho(k)}^t < 1 \) for all \( k \geq K \).

Now fix \( \delta_0 > 0 \) sufficiently small so that \( \delta_0 < s_K \), and let \( \delta \in (0, \delta_0) \). Let \( n \in \mathbb{N} \) be such that
\[
s_{n+1} \leq \delta < s_n.
\]

First, suppose that \( \gamma(n) \leq \rho(n) < n \). Then \( \rho(n) \in [\delta^{1/\theta}, \delta] \). 
Removing from $[0,1]$ the complementary gaps of $E$ that have lengths $a_1,...,a_{2^{\rho(n)} - 1}$, we obtain the set \( E_1\cup E_2\cup \dots \cup E_{2^{\rho(n)}} \),where $E_i$ are closed intervals (some of these intervals may be degenerate, that is, single points).

Then, the collection \( \{E_i\} \) forms a cover of \( E \), and we have
\[
\sum_{i=1}^{2^{\rho(n)}} |E_i| = \sum_{i=2^{\rho(n)}}^{\infty} a_i = 2^{\rho(n)} s_{\rho(n)}.
\]
By Remark~\ref{lemma cubrimiento}, it follows that there exists a family of \( 2^{\rho(n)+1} \) sets, \( \{U_i\} \), such that
\[
E \subset \bigcup_{i=1}^{2^{\rho(n)}} E_i \subset \bigcup_{i=1}^{2^{\rho(n)+1}} U_i,
\quad \text{
with \( |U_i| = s_{\rho(n)} \) for all \( i \)}.
\]
Then we have
\[
\sum_i |U_i|^t = 2 \cdot \left( 2^{\rho(n)} s_{\rho(n)}^t \right) < 2,
\]
which implies \( \uda E \leq t \), and the result follows by letting \( t \to \uda C_a \).

Now, if \( \rho(n) = \gamma(n) + 1 \) or \( \rho(n) = n \), then either \( s_{\rho(n)-1} \) or \( s_{\rho(n)+1} \) belongs to \( [\delta^{1/\theta}, \delta] \), and the result follows by noting that the hypothesis 
\(
\lim_{n \to \infty} \left( \frac{s_n}{s_{n+1}} \right)^{1/n} = 1
\)
implies
\begin{align*}
\limsup_{n \to \infty} \frac{-\rho(n)\log 2}{\log s_{\rho(n)-1}}  
&= \limsup_{n \to \infty} \frac{-\rho(n)\log 2}{\log s_{\rho(n)+1}}    
\\ &= \limsup_{n \to \infty} \frac{-\rho(n)\log 2}{\log s_{\rho(n)}} 
= \uda C_a.
\end{align*}
\end{proof}
 
\section{Attainable values}
A natural question is whether all possible values of $\theta$-intermediate dimensions for sets in $\mathcal{C}_a$ are attained within the intervals $[\underline{\dim}_\theta D_{a}, \underline{\dim}_\theta C_{a}]$ or $[\overline{\dim}_\theta D_{a}, \overline{\dim}_\theta C_{a}]$. In the next theorem, we provide conditions under which this holds.

Recall that the Assouad dimensions of complementary sets form a closed interval if the sequence is doubling. Conversely, the lower Assouad dimensions of complementary sets form a closed interval if the sequence satisfies $2 < \inf_n s_n/s_{n+1}$ (see \cite{garcia2018assouad}). These two conditions, in fact, constitute our hypotheses to ensure that the intermediate dimensions of complementary sets also form a closed interval. Indeed, it is straightforward to prove that for a decreasing sequence satisfying $2 < \inf_n s_n/s_{n+1}$, the doubling condition is equivalent to $\sup_n s_n/s_{n+1} < \infty$.

\begin{theo}\label{Maintheo}
Let \( \theta \in (0,1] \) and let \( a = \{a_n\} \) be a positive, non-increasing, summable sequence. If 
\[
2 < \inf_n \frac{s_n(a)}{s_{n+1}(a)} \leq \sup_n \frac{s_n(a)}{s_{n+1}(a)} < \infty,
\]
then for every 
\[
t \in \left[ \overline{\dim}_\theta D_a, \, \overline{\dim}_\theta C_a \right],
\]
there exists a set \( E = E(t, \theta) \in \mathcal{C}_a \) such that \( \overline{\dim}_\theta E = t \). An analogous result holds for the upper $\theta$-intermediate dimension, i.e., for every 
\[
t \in \left[ \underline{\dim}_\theta D_a, \, \underline{\dim}_\theta C_a \right],
\]
there exists a set \( E' = E'(t, \theta) \in \mathcal{C}_a \) such that \( \underline{\dim}_\theta E' = t \).
\end{theo}

\begin{proof}
The proof proceeds by explicit construction of the sets $E$ and $E'$. Given the sequence \(a\), we will construct a Cantor-type set \(C_b\) associated to a subsequence \(b\) of \(\{a_n\}\), and a countable set \(D_{a'}\) associated to the remaining elements of \(\{a_n\}\). We will then define
\[
E = C_b \cup D_{a'},
\]
and show that \( \overline{\dim}_\theta E = t\).

First, we note that the hypothesis $2<\inf_n\frac{s_n(a)}{s_{n+1}(a)}\leq \sup_n\frac{s_n(a)}{s_{n+1}(a)}<\infty $ implies that $a$ is a doubling sequence.

Let \( \theta \in (0,1] \) and fix \( t \in ( \overline{\dim}_\theta D_a , \overline{\dim}_\theta C_a) \). Define
\[
s := \frac{\overline{\dim}_\theta C_a}{t} > 1.
\]
For each \(n\), let \(j=j(n)\) be such that
\begin{equation}\label{j(n)}
    a_{2^{j+1}} < (a_{2^n})^s \leq a_{2^j}, \qquad j \geq n.
\end{equation}

\[ b_{2^k+i} = \begin{cases}                  a_{2^{j(k)+1}+i} & \text{if $j(k)>j(k-1)$,}\\                  a_{2^{j(k)+1}+\sum_{r=1}^m 2^{k-r}+i} &    \text{if $j(k)=\cdot\cdot\cdot=j(k-m)>j(k-m-1)$,}           \end{cases} \]

for \(k\ge 1\) and \(0\le i<2^k\), where \(m\) denotes the largest integer with \(0\le m\le k-1\) such that \(j(k)=j(k-1)=\cdots=j(k-m)\). 
We also adopt the convention \(j(0)=0\).\\
By construction and using the doubling condition of $a$, there exists a constant $c>0$ such that for each \(k \geq 1\) and \(0 \leq i < 2^k\),
\[
c^{-1} (a_{2^k})^s \leq b_{2^k+i} \leq c (a_{2^k})^s.
\]

Now, observe that

$$
\inf_n \frac{s_n(b)}{s_{n+1}(b)} > 2.
$$

Indeed, recall that

$$
s_n(a) = 2^{-n} \sum_{i=2^n}^{2^{n+1}-1} a_i + 2 s_{n+1}(a).
$$

From this identity, the hypothesis

$$
2 < \inf_n \frac{s_n(a)}{s_{n+1}(a)}
$$

can be rewritten in the equivalent form

$$
\sup_n \frac{s_{n+1}(a)}{r_n(a)} < \infty,
$$

where

$$
r_n(a) = 2^{-n}\sum_{i=2^n}^{2^{n+1}-1} a_i.
$$

Using the doubling condition for $a$, one directly obtains

$$
\frac{s_{n+1}(b)}{r_n(b)}\leq \frac{C}{2^{-(n+1)}} \sum_{i=2^{n+1}}^{\infty}\left(\frac{a_i}{a_{2^n}}\right)^s
< C\, \frac{2^{-(n+1)}\sum_{i=2^{n+1}}^{\infty}a_i}{a_{2^n}}
\leq C'\frac{s_{n+1}(a)}{r_n(a)},
$$

for some constant $C$ and $C'$, and hence $\sup_n \tfrac{s_{n+1}(b)}{r_n(b)} < \infty$.

It follows that for doubling decreasing sequences with this property we have\[b_{2^n}\asymp s_n(b).\]Indeed, if \(s_{n+1}(b)/s_n(b)<1/2-\delta\), then\[2\delta s_n(b) \leq s_n(b)-2s_{n+1}(b)=2^{-n}\sum_{j=2^n}^{2^{n+1}-1}b_j\leq b_{2^{n}}\leq c\, 2^{-n} \sum_{j=2^n}^{2^{n+1}-1}b_j \leq c\,s_n(b).\]Hence\[s_n(b)\asymp b_{2^n}\asymp (a_{2^n})^s\asymp (s_n(a))^s.\]

Now, we will prove that $\overline{\dim}_\theta C_b=s^{-1}\overline{\dim}_\theta C_a.$

Let $d > \overline{\dim}_\theta C_a$. Then, for all $\epsilon > 0$, there exists $\delta_0 > 0$ such that for every $\delta \in (0,\delta_0)$ there is a cover $\{U_i\}$ of $C_a$ satisfying
\begin{equation}\label{ecuacion definicion}
\delta \leq |U_i| \leq \delta^\theta
\quad \text{and} \quad
\sum_i |U_i|^d < \epsilon.
\end{equation}
Let $\delta$ and the cover $\{U_i\}$ satisfy \eqref{ecuacion definicion}. For each $i$, choose $k(i)$ so that
\[
s_{k(i)+1}(a) \leq |U_i| < s_{k(i)}(a).
\]

It follows that each $U_i$ can intersect at most four closed intervals from step $k(i)-2$ in the construction of $C_a$, namely the intervals $I_{j_n}^{\,k(i)-2}(a)$ with $n=1,\dots,4$. Consequently, we can cover $C_a$ by
\[
C_a \subset \bigcup_i \bigcup_{n=1}^4 I_{j_n}^{\,k(i)-2}(a).
\]
Analogously, we can cover $C_b$ by selecting the corresponding intervals from the same steps in the construction of $C_b$, that is,
\[
C_b \subset \bigcup_i \bigcup_{n=1}^4 I_{j_n}^{\,k(i)-2}(b).
\]

Now, since
\[
\delta^s \leq c_1 s_{k(i)-3}(b) \leq c_1 |I_{j_n}^{\,k(i)-2}(b)| \leq c_1 s_{k(i)-1}(b) 
\leq c_1 c_2 \,(s_{k(i)-1}(a))^s \leq c_1 c_2\, |U_i|^s,
\]
for some constants $c_1, c_2$, it follows that
\[
|I_{j_n}^{\,k(i)-2}(b)| \in \bigl(c_1^{-1}\delta^s,\, c_1 c_2 \delta^{\theta s}\bigr).
\]
Moreover,
\[
\sum_i \sum_{n=1}^4 |I_{j_n}^{\,k(i)-2}(b)|^{d/s} < 4\, c_1 c_2 \,\epsilon,
\]
and therefore
\[
\overline{\dim}_\theta C_b \leq s^{-1}\,\overline{\dim}_\theta C_a.
\]
An analogous argument shows that
\[
s^{-1}\,\overline{\dim}_\theta C_a \leq \overline{\dim}_\theta C_b.
\]

We now define the complementary countable part. Let $\{\tilde{a}_n\}$ be a non-increasing sequence such that
\[
\{a_n : n \in \mathbb{N}\} = \{b_n : n \in \mathbb{N}\} \cup \{\tilde{a}_n : n \in \mathbb{N}\}.
\]
Define the associated countable set
\[
D_{a'} = \bigcup_{k=1}^{\infty} y_k, \qquad y_k = r + \sum_{i=k}^\infty \tilde{a}_i,
\]
where
\[
r = \sum_{i=1}^{\infty} b_i.
\]

Finally, set
\[
E = C_b \cup D_{a'}.
\]
Then \(E \in \mathcal{C}_a\). By finite stability and the invariance of the upper box dimension for complementary sets we have
\[
\ubd D_a =\overline{\dim}_B E = \max\{\overline{\dim}_B C_b, \overline{\dim}_B D_{a'}\}.
\]
Since
\[
\overline{\dim}_B C_b = s^{-1}\overline{\dim}_B C_a < \overline{\dim}_B D_a,
\]
we deduce that
\[
\overline{\dim}_B D_a = \overline{\dim}_B D_{a'}.
\]
Therefore
\[
\overline{\dim}_\theta D_a = \overline{\dim}_\theta D_{a'},
\]
by Theorem \ref{dimension de suceciones}.

Combining these results, we conclude that
\begin{align*}
\overline{\dim}_\theta E
&= \max\{s^{-1}\overline{\dim}_\theta C_a , \overline{\dim}_\theta D_{a'}\} \\
&= \max\{t, \overline{\dim}_\theta D_a\} = t,
\end{align*}
as required.

For the lower $\theta$-intermediate dimension, we must change strategy, as the lower $\theta$-intermediate dimension (as well as the lower box dimension) is not finitely stable. We will consider a similar set and proceed with a covering strategy to prove the theorem.

Let \( t \in (\underline{\dim}_\theta D_a, \underline{\dim}_\theta C_a) \). Proceed as before to construct a set \( E' \in \mathcal{C}_a \) such that \( E' = C_b \cup D_{a'} \), where $C_b$ is a decreasing Cantor set and \( D_{a'} \) is a countable set.

Since both sequences \( \{b_n\} \) and \( \{a_n\} \) are non-increasing, it follows from Proposition \ref{hausdorff cantor} and the subsequent remarks that \[ s^{-1}\underline{\dim}_\theta C_a = s^{-1}\underline{\dim}_B C_a = \underline{\dim}_B C_b = \underline{\dim}_\theta C_b  \] for all \( \theta \).

Now since  
\[
\frac{\theta \underline{\dim}_B C_a}{1 - \underline{\dim}_B C_a (1 - \theta)} = \underline{\dim}_\theta D_a < t = \frac{\underline{\dim}_B C_a}{s},
\]  
we have that  
\[
s < \frac{1 - (1 - \theta) \underline{\dim}_B C_a}{\theta}.
\]  
Therefore, we can choose \( t' > t \) and \( B > \underline{\dim}_B C_a > b \) such that  
\begin{equation}\label{condicion s}
t' > \max\left\{ \frac{B}{s},\ \frac{\theta B}{1 - B(1 - \theta)},\ \frac{\theta b - \theta(1 - b / t's)}{1 - b(1 - \theta) - (1 - b / t's)(1 - \theta)} \right\}
\end{equation}  
and  
\begin{equation}\label{condicion s 2}
s < \frac{1 - (1 - \theta) s t'}{\theta}.
\end{equation}

Now, since by construction of $E'$, $(s_n(a))^s \asymp s_n(b)$, it follows that there exist a subsequence $\{n_k\}$ and a constant $C > 0$ such that
\begin{equation}\label{Desigualdad de sna}
2^{-n_k(1/b)} \leq s_{n_k}(a) \leq 2^{-n_k(1/B)}
\end{equation}
and
\begin{equation}\label{Desigualdad de snb}
s_{n_k}(b) \leq C \, 2^{-n_k(s/B)}.
\end{equation}

Let $\delta_k := (s_{n_k}(a))^{1/(t' + \theta(1 - t'))}$. 
By equation \eqref{condicion s 2}, we have 
\[
(s_{n_k}(a))^s = \delta_k^{s t'(1 - \theta) + \theta s} \in (\delta_k, \delta_k^\theta),
\]
and therefore $C' s_{n_k}(b)$ is an admissible diameter for covering sets, for some constant $C'$.

Since the closed intervals of a decreasing Cantor set, say \( I_i^n \), satisfy \( s_{n+1} \leq |I_i^n| \leq s_{n-1} \), we can cover \( C_b \) with \( 2^{n_k + 1} \) intervals of length \(C' s_{n_k}(b) \).

To cover \( D_{a'} = \{x'_1, x'_2, \ldots\} \), we will cover \( \{x'_i : i = 1, \ldots, 2^{n_k}\} \) with \( 2^{n_k} \) sets of diameter \( \delta \), and the remaining points will be covered with
\[
\left\lceil \frac{2^{n_k} s_{n_k}(a'')}{\delta^\theta} \right\rceil \leq \left\lceil \frac{2^{n_k} s_{n_k}(a)}{\delta^\theta} \right\rceil \leq \frac{2^{-n_k(1 - t's)/t's}}{\delta^\theta} + 1,
\]
sets of diameter \( \delta^\theta \), using \eqref{Desigualdad de sna}.

Therefore we have created a cover \( \{U_i\} \), of \( E \), by sets whose diameters belong to \( [\delta, \delta^{\theta}] \), and that satisfies
\begin{align*}
    \sum_i |U_i|^{t'} &\leq 2C' 2^{n_k} s_{n_k}(b)^{t'} + 2^{n_k} \delta^{t'} + \left( c' \frac{2^{-n_k(1 - t's)/t's}}{\delta^\theta} + 1 \right) \delta^{\theta t'} \\
    &= S_1 + S_2 + S_3.
\end{align*}
Using \eqref{Desigualdad de snb} and \eqref{condicion s}, we have that \( S_1 \to 0 \) as \( n_k \to \infty \).

For \( S_2 \) we have
\begin{align*}
    2^{n_k} \delta^{t'} = 2^{n_k} (s_{n_k}(a))^{t'/(\theta + t'(1 - \theta))} \leq 2^{-n_k(t'/(B(\theta + t'(1 - \theta))) - 1)} \to 0
\end{align*}
since \( t' > \theta B / (1 - B(1 - \theta)) \).

On the other hand, \( S_3 \) will tend to \( 0 \) since by \eqref{condicion s} we have 
\[
t' > \frac{\theta b - \theta(1 - b / t's)}{1 - b(1 - \theta) - (1 - b / t's)(1 - \theta)}.
\]

Then, it follows that
\[
\underline{\dim}_\theta E' \leq t'
\]
and the result follows by letting \( t' \to t \) and noting that \( t \leq \underline{\dim}_\theta E' \) since \( C_b \subset E' \).
\end{proof}

\section*{Further Work}

In \cite{Banaji-R22}, the authors provided necessary and sufficient conditions for a function $h(\theta)$ to arise as the intermediate dimension of a bounded subset of $\mathbb{R}^n$. Specifically, for $0 \leq \lambda < \alpha \leq n$, denote by $\mathcal{H}(\lambda,\alpha)$ the set of all functions $h:[0,1]\to [\lambda,\alpha]$ that are non-decreasing, continuous on $(0,1]$, and satisfy
\[
\limsup_{\epsilon\to 0^+}\frac{h(\theta+\epsilon)-h(\theta)}{\epsilon}\leq \frac{(h(\theta)-\lambda)(\alpha - h(\theta))}{(\alpha-\lambda)\theta}, \quad \forall\, \theta\in (0,1).
\]
If $\lambda = \alpha$, then $\mathcal{H}(\lambda,\alpha)$ consists only of the constant function $h(\theta)=\lambda$.  
The following theorem was established in \cite{Banaji-R22}.

\begin{theo}\cite[Theorem B]{Banaji-R22}\label{BanajiRutar}
Let $F\subset \mathbb{R}^n$ satisfy $\dim_L F=\lambda$ and $\dim_A F=\alpha$. Then with $h_1(\theta)=\da F$ and $h_2(\theta)=\uda F$, we have $h_1,h_2\in \mathcal{H}(\lambda,\alpha)$, $h_1\leq h_2$, and $h_1(0)=h_2(0)$.  

Conversely, if $0\leq \lambda\leq \alpha\leq n$ and $h_1,h_2\in \mathcal{H}(\lambda,\alpha)$ satisfy $h_1\leq h_2$ and $h_1(0)=h_2(0)$, then there exists a compact perfect set $F\subset \mathbb{R}^n$ such that
\[
\dim_L F=\lambda,\qquad 
\da F=h_1(\theta),\qquad 
\uda F=h_2(\theta),\qquad 
\dim_A F=\alpha,
\]
for all $\theta\in[0,1]$.
\end{theo}

A natural question connecting \cite{Banaji-R22} with the present work is the following:  
for a fixed sequence $a$, can we describe conditions ensuring that a given function $h(\theta)$ can be realized as the intermediate dimension of a complementary set $E\in \mathcal{C}_a$?  

In this paper, some partial results in that direction were obtained. For instance, under the assumption 
\[
2<\inf_n \frac{s_n}{s_{n+1}}\leq \sup_n \frac{s_n}{s_{n+1}}<\infty,
\]
the upper intermediate dimension part of Theorem~\ref{Maintheo} can be interpreted as follows:  
for any $s>1$, there exists a set $E\in \mathcal{C}_a$ such that
\[
\uda E=\max \left\{ \frac{\uda C_a}{s},\,
\frac{\theta\,\ubd C_a}{1-(1-\theta)\,\ubd C_a} \right\}
=:h_s(\theta),
\]
that is, for each $s>1$, the function $h_s(\theta)$ is attained as the upper intermediate dimension of some complementary set.
Moreover, any dimension function $h(\theta)$ of a set $E\in \mathcal{C}_a$, must satisfy
\[
\frac{\theta\,\ubd C_a}{1-(1-\theta)\,\ubd C_a}\leq h(\theta)\leq \uda C_a.
\]

To conclude, we state explicitly the open problem that naturally emerges from this discussion.

\begin{question}
Given a non-increasing, summable sequence $a=\{a_n\}$ satisfying 
\[
2<\inf_n\frac{s_n(a)}{s_{n+1}(a)}\leq \sup_n\frac{s_n(a)}{s_{n+1}(a)}<\infty,
\]
describe the structure of the families of functions
\[
\mathcal{F}^U(a):=\bigl\{\,\theta\mapsto \uda E : E\in\mathcal{C}_a\,\bigr\}
\quad\text{and}\quad
\mathcal{F}^L(a):=\bigl\{\,\theta\mapsto \da E : E\in\mathcal{C}_a\,\bigr\},
\]
that is, the collections of all functions which can arise respectively as the upper and lower $\theta$-intermediate dimensions of complementary sets associated with $a$.
\end{question}

\section*{Funding}
The research of the authors is partially supported by Grants
PICT 2022-4875 (ANPCyT) (no disbursements since December 2023), 
PIP 202287/22 (CONICET), and 
UBACyT 2022-154 (UBA). Nicolas Angelini is also partially supported by PROICO 3-0720 ``An\'alisis Real y Funcional. Ec. Diferenciales''.

\end{document}